\documentclass[11pt, reqno]{amsart}
\usepackage{amsmath, amsfonts, amssymb}
\usepackage[left=1.5in, right=1.5in, top=1in, bottom=1in]{geometry}
\usepackage[dvipsnames, svgnames]{xcolor}
\usepackage[inline]{asymptote}
\usepackage{adjustbox, amssymb, enumitem, tikz-cd, multirow, mathtools}

\usepackage[pdfusetitle, pdfencoding=auto, psdextra, breaklinks=true]{hyperref}
\usepackage{xcolor}

\usepackage[nameinlink]{cleveref}

\renewcommand{\textbf}[1]{{\bfseries\boldmath #1}}
\newcommand{\vocab}[1]{\textbf{\textcolor{BrickRed}{\boldmath #1}}}

\newcommand{\calA}{\mathcal A}
\newcommand{\calB}{\mathcal B}
\newcommand{\calF}{\mathcal F}
\newcommand{\calG}{\mathcal G}
\newcommand{\calH}{\mathcal H}
\newcommand{\calM}{\mathcal M}
\newcommand{\calS}{\mathcal S}
\newcommand{\calT}{\mathcal T}
\newcommand{\calU}{\mathcal U}
\newcommand{\eps}{\varepsilon}
\newcommand{\bbP}{\mathbb P}

\definecolor{mylinkcolor}{rgb}{0.0,0.0,0.7}
\definecolor{myurlcolor}{rgb}{0.0,0.0,0.7}
\hypersetup{
 colorlinks,
 urlcolor=myurlcolor,
 citecolor=myurlcolor,
 linkcolor=mylinkcolor,
 breaklinks=true
}

\usepackage{thmtools}
\declaretheorem{theorem}[name=Theorem, numberwithin=section]
\declaretheorem{lemma}[name=Lemma, sibling=theorem]
\declaretheorem{claim}[name=Claim, sibling=theorem]
\declaretheorem{proposition}[name=Proposition, sibling=theorem]
\declaretheorem{corollary}[name=Corollary, sibling=theorem]
\declaretheorem{definition}[style=definition, name=Definition, sibling=theorem]
[style=definition, name=Remark, sibling=theorem]
[style=definition, name=Example, sibling=theorem]
\declaretheorem{openprob}[style=definition, name=Open Question, sibling=theorem]
\declaretheorem{theorem*}[name=Theorem, numbered=no]

\usepackage[backend=bibtex, style=alphabetic, sorting=nyt,
maxnames=20, maxalphanames=20, backref, isbn=false, url=true, doi=true]{biblatex}
\DefineBibliographyStrings{english}{%
  backrefpage = {$\uparrow$ p.},%
  backrefpages = {$\uparrow$ pp.}%
}
\defbibheading{bibliography}[\refname]{\section*{#1}}

\title{Improved Bounds on Rainbow $k$-partite Matchings}
\author{Pitchayut Saengrungkongka}
\address{
  Department of Mathematics,
  Harvard University,
  Cambridge,
  MA 02138,
  USA
}
\email{saengrun@math.harvard.edu}

\begin{document}
\begin{abstract}
Let $n$, $s$, and $k$ be positive integers.
We say that a sequence $f_1,\dots,f_s$ of nonnegative integers
is \emph{satisfying} if for any collection of $s$ families 
$\mathcal F_1,\dots,\mathcal F_s\subseteq [n]^k$
such that $|\mathcal F_i|>f_i$ for all $i$,
there exists a rainbow matching, i.e.,
a list of pairwise disjoint tuples 
$F_1\in\mathcal F_1$, \dots, $F_s\in\mathcal F_s$.
We investigate the question, posed by Kupavskii 
and Popova, of determining the smallest $c=c(n,s,k)$
such that the arithmetic progression 
$c$, $n^{k-1}+c$, $2n^{k-1}+c$, \dots, $(s-1)n^{k-1}+c$ is satisfying.
We prove that the sequence is satisfying if
$c\geq \Omega_k(\max(s^2n^{k-2}, sn^{k-3/2}\sqrt{\log s}))$,
improving the previous result by Kupavskii and Popova.
We also study satisfying sequences for $k=2$ using the polynomial method,
extending the previous result by Kupavskii and Popova 
to when $n$ is not prime.
\end{abstract}
\maketitle

\section{Introduction}
One of the most basic open problems in extremal set theory
is the \textbf{Erd\H os matching conjecture}
\cite{erdos_matching}, which states that if $n\geq ks$ and
$\calF\subseteq \binom{[n]}k$ (see \Cref{subsec:notation}
for notation) satisfies 
$$|\calF| > t(n,s,k) \coloneq \max\left(\binom{sk-1}k, 
\binom nk - \binom{n-s+1}k\right),$$
then there exist pairwise disjoint 
sets $F_1,\dots,F_s\in\calF$.
In other words, this conjecture asks 
whether $t(n,s,k)$ is the maximum number of hyperedges 
in a $k$-uniform hypergraph without $s$ disjoint hyperedges.
Note that the bound $t(n,s,k)$ cannot be improved 
because neither family $\calF = \binom{[sk-1]}{k}$
nor $\calF = \binom{[n]}{k} \setminus \binom{[n-s+1]}{k}$
contain such pairwise disjoint $F_1,\dots,F_s$.
This conjecture is known to hold when $k=2$ \cite[Thm.~4]{emc_2},
when $k=3$ \cite{emc_3}, and when $n\geq \frac 53sk$
and $s$ is sufficiently large \cite[Thm.~1]{emc_bound},
but the general case remains open.

Huang, Loh, and Sudakov \cite{rainbow_emc} introduced 
the rainbow version of the Erd\H os matching conjecture,
which states that if $\calF_1,\dots,\calF_s\subseteq \binom{[n]}k$
are such that $|\calF_i| > t(n,s,k)$ for all $i\in [s]$,
then there exist pairwise disjoint sets 
$F_1\in\calF_1,\dots, F_s\in\calF_s$.
In other words, instead of having a single hypergraph,
we have a hypergraph colored in $s$ colors
(a hyperedge can have multiple colors),
and we want a rainbow matching, i.e., $s$ disjoint 
edges of different colors.
This conjecture was proved when $k=2$
\cite[Thm.~3.1]{rainbow_k_partite_emc}
and when $n>Csk$ for 
some large constant $C$ \cite[Thm.~1.2]{rainbow_emc_bound},
but the general case remains open.

One can consider a $k$-partite analogue
of these two conjectures.
Two tuples $F,G\in [n]^k$ are \vocab{disjoint}
if the $i$-th coordinate of $F$ is different 
from the $i$-th coordinate of $G$
for each $i\in [k]$. 
Then the $k$-partite analogue of the Erd\H os matching conjecture
states that given $\calF\subseteq [n]^k$
such that $|\calF| > (s-1)n^{k-1}$,
there exist $s$ pairwise disjoint tuples 
$F_1,\dots,F_s\in\calF$.
This has an easy proof via an averaging argument:
a random matching in $[n]^k$ is expected to 
intersect $\calF$ in $|\calF|/n^{k-1} > s-1$ elements,
so there exists a matching that intersects $\calF$
in at least $s$ elements.
The bound $(s-1)n^{k-1}$ cannot be improved 
because one can take $\calF = [s-1]\times [n]^{k-1}$,
which clearly has no such pairwise disjoint tuples.

Aharoni and Howard \cite{rainbow_k_partite_emc}
introduced the rainbow version of this problem,
which says that if $\calF_1,\dots,\calF_s\subseteq [n]^k$
satisfy $|\calF_i| > (s-1)n^{k-1}$ for all $i$,
then there exist $s$ pairwise disjoint tuples 
$F_1\in\calF_1,\dots,F_s\in\calF_s$.
This conjecture was proved for $s\geq 470$ 
by Kiselev and Kupavskii \cite[Thm.~1]{rainbow_matching}.
They also considered an asymmetric version of this problem,
where each subset is given a separate lower bound.
\begin{definition}
Let $n\geq s$ and $k$ be positive integers.
Then the sequence $f_1,\dots,f_s$ is \vocab{satisfying}
if for every list of subsets $\calF_1,\dots,\calF_s\subseteq [n]^k$
with $|\calF_i| > f_i$,
there exists a list of $s$ pairwise disjoint elements 
$F_1\in\calF_1,\dots, F_s\in\calF_s$.
Such a list is called a
\vocab{rainbow matching} of $\calF_1,\dots,\calF_s$.
\end{definition}
Indeed, Kiselev and Kupavskii \cite[Thm.~1]{rainbow_matching}
showed that the arithmetic sequence $f_i  = (i-1 + C\sqrt{s\log s})n^{k-1}$
is satisfying for some absolute constant $C>0$.
Their proof takes the intersection
between a uniformly random matching and $\calF_i$
and uses the fact that the size of this intersection 
has subgaussian concentration.

In a sequel paper, Kupavskii and Popova 
\cite{satisfying_sequence} proved various results 
about satisfying sequences, one of which is the following theorem.
\begin{theorem}[{\cite[Thm.~9]{satisfying_sequence}}]
If $n>2^5s\log_2(sk)$ and 
$$c\geq 4s^2n^{k-2} + 2^{15}s^3\log_2(ks)^3 n^{k-3},$$
then the arithmetic sequence 
$f_i = (i-1)n^{k-1}+c$
is satisfying.
\end{theorem}
Their proof uses the method of spread approximation 
(first introduced by Kupavskii and Zakharov \cite{spread_approx})
to reduce the problem to finding matchings 
in a certain auxilary family of sets,
where each set has cardinality at most two.
Then they conclude by an elementary argument.
The following question arises naturally from the above theorem.
\begin{openprob}
\label{prob:ap_satisfying}
Given $n,s,k$,
determine the smallest $c=c(n,s,k)$ such that the 
sequence $f_i=(i-1)n^{k-1}+c$ is satisfying.
\end{openprob}
\subsection{Our Results}
In this paper, we combine ideas from the proofs of 
\cite[Thm.~1]{rainbow_matching} and \cite[Thm.~9]{satisfying_sequence}
to obtain a better upper bound for $c$.
Our results come in two different forms;
one has a dependence on $k$ and one does not.
\begin{theorem}
\label{thm:with_k}
If $n> s$ and
$$c\geq n^{k-1} + \max\left(k^2sn^{k-\frac 32}\sqrt{8\log(2ks)}, 
8kn^{k-1}\log(2ks)\right),$$
then the sequence $f_i = (i-1)n^{k-1}+c$
is satisfying.
\end{theorem}
\begin{theorem}
\label{thm:without_k}
If $n > \max(2^5 s\log_2(sk), s)$ and 
$$c \geq n^{k-1} + \max\left(14sn^{k-\frac 32}\sqrt{\log(2ks)},
8kn^{k-1}\log(2ks)\right) + 2^{15} s^3\log_2(ks)^3 n^{k-3},$$
then the sequence $f_i = (i-1)n^{k-1}+c$ is satisfying.
\end{theorem}

Notice that when $s$ is large, namely
$s\gg n^{3/4+\eps}$, the term $2^{15}s^3\log_2(ks)^3n^{k-3}$
dominates the bound in \Cref{thm:without_k},
so \Cref{thm:with_k} works better.
However, when $s\ll n^{3/4-\eps}$, \Cref{thm:without_k}
works better, but the constant factor in front of 
$sn^{k-3/2} \sqrt{\log(2ks)}$ has dependency on $k$.

To help interpret the result, we hold $k$ constant 
and consider the relative magnitudes of $s$ and $n$.
The best bounds in different regimes of $(s,n)$ 
are described in 
\Cref{table:results}.
We believe that the upper bounds are very unlikely to be tight.
\begin{table}[h]
\begin{tabular}{ccc}
Condition & Best upper bound on $c$
& Best lower bound on $c$ \\ \hline
$s \ll_k n^{1/2-\eps}$ & 
\begin{tabular}{c}
$O_k(s^2 n^{k-2})$   \\
{\footnotesize (\cite[Thm.~9]{satisfying_sequence})}
\end{tabular} &
\begin{tabular}{c}
$\Omega_k(s^2 n^{k-2})$   \\
{\footnotesize (\cite[Claim 7]{satisfying_sequence})}
\end{tabular}
\\
$n^{1/2+\eps}\ll_k s \ll_k n^{3/4-\eps}$ 
&
\begin{tabular}{c} 
$O_k(sn^{k-3/2}\sqrt{\log(2ks)})$ \\
{\footnotesize (\Cref{thm:with_k} or \ref{thm:without_k})} 
\end{tabular} & 
\begin{tabular}{c}
$\Omega_k(n^{k-1})$   \\
{\footnotesize (\cite[Claim 7]{satisfying_sequence})}
\end{tabular}
\\
$n^{3/4+\eps}\ll_k s \ll_k n$ 
& 
\begin{tabular}{c}
$O_k(sn^{k-3/2}\sqrt{\log(2ks)})$ \\ 
{\footnotesize (\Cref{thm:with_k})} 
\end{tabular}& 
\begin{tabular}{c}
$\Omega_k(n^{k-1})$   \\
{\footnotesize (\cite[Claim 7]{satisfying_sequence})}
\end{tabular}
\\
\end{tabular}
\caption{Best bounds for \Cref{prob:ap_satisfying}
when $k$ is held constant.}
\label{table:results}
\end{table}

We also extend the polynomial method result from 
\cite[Thm.\ 3]{satisfying_sequence} for non-prime $n$,
resulting in the following theorem.
\begin{theorem}
\label{thm:polynomial_method}
Let $k=2$, and
let $(a_1,\dots,a_s)$ and $(b_1,\dots,b_s)$ be two permutations of 
$\{0,1,\dots,s-1\}$.
Then the sequence $f_i = n(a_i+b_i)$ is satisfying.
\end{theorem}
In comparison, the following result only works when $n$ is prime:
\begin{theorem}[{\cite[Thm.~3]{satisfying_sequence}}]
\label{thm:old_combonull}
Let $k=2$, $n=p$ be prime, and $e_1,\dots,e_s$ be a sequence 
of nonnegative integers such that the coefficient of 
$x_1^{e_1}\cdots x_s^{e_s}$ in
$\prod_{1\leq i<j\leq s}(x_i-x_j)^2$
is nonzero modulo $p$.
Then the sequence $f_i = ne_i$ is satisfying.
\end{theorem}
Note that \Cref{thm:old_combonull} is a special case of 
\Cref{thm:polynomial_method}.
In particular, by considering Laplace expansion 
of the Vandermonde determinant, we can see that 
$\prod_{1\leq i<j\leq s} (x_i-x_j)$
is a linear combination of monomials of the form 
$x_1^{a_1}\cdots x_s^{a_s}$, where $(a_1,\dots,a_s)$
is a permutation of $\{0,1,\dots,s-1\}$.
Thus, if a monomial $x_1^{e_1}\cdots x_s^{e_s}$
has a nonzero coefficient in $\prod_{1\leq i<j\leq s}(x_i-x_j)^2$,
then $e_i = a_i+b_i$, for some permutations $(a_1,\dots,a_s)$
and $(b_1,\dots,b_s)$ of $\{0,1,\dots,s-1\}$.
Furthermore, \Cref{thm:old_combonull} only works when 
$n=p$ is prime and require the coefficient to be 
nonzero modulo $p$ (instead of nonzero).

\Cref{thm:polynomial_method} implies that 
the sequence $(s-1)n$, \dots, $(s-1)n$ and 
the sequence $0$, $2n$, $4n$, \dots, $2(s-1)n$ are 
both satisfying.
Note that these two sequences are far from optimal,
especially when $s$ is very small 
since we know from \cite[Thm.~9]{satisfying_sequence} that
the sequence $c, n+c, 2n+c,\dots, (s-1)n+c$
is satisfying when $c\geq 4s^2$.%
\footnote{When $k=2$, there is no term $2^{15}s^3\log_2(ks)^3n^{k-3}$
because one can skip the spread approximation step 
and proceed directly to the elementary argument after Lemma 10.}
The key input for \Cref{thm:polynomial_method} is the 
multivariate generalization \cite{multivar_null}
of Alon's combinatorial nullstellensatz \cite{combonull}.

\subsection{Outline}
We prove \Cref{thm:with_k} in \Cref{sec:with_k}
and prove \Cref{thm:without_k} in \Cref{sec:without_k}.
The proofs of these two theorems are similar,
except that in the proof of \Cref{thm:without_k}, 
one replaces an elementary argument with
the method of spread approximation to eliminate 
the factor of $k$ in the bound.
We prove \Cref{thm:polynomial_method} in \Cref{sec:polynomial_method}.

\subsection*{Acknowledgements}
This research was conducted at 
the University of Minnesota Duluth REU 
with support from Jane Street Capital, NSF Grant 2409861,
and donations from Ray Sidney and Eric Wepsic.
We thank Colin Defant and Joe Gallian for such 
a wonderful opportunity.
We also thank Evan Chen, Noah Kravitz, Rupert Li, Maya Sankar,
and Daniel Zhu for helpful discussions and feedback on this paper.
We also thank Albert Wang for a discussion that led the 
author to realize an error in \Cref{sec:polynomial_method}.

\section{Preliminaries}
\subsection{Notation}
\label{subsec:notation}
We let $[n] \coloneq \{1,2,\dots,n\}$ denote the standard 
$n$-element set. 
For any set $X$, we let $\binom{X}k$ 
be the set of all $k$-element subsets of $X$,
and let $X^k$ be the set of all $k$-tuples
whose elements are in $X$.

Given multisets $A$ and $B$,
their sum $A\oplus B$ is the multiset such that  
for any element $x$ appearing in $A$ and $B$
$a$ and $b$ times, respectively,
$x$ appears in $A\oplus B$ exactly $a+b$ times.

In \Cref{sec:with_k} and \Cref{sec:without_k},
it is helpful to view a tuple in
$[n]^k$ as a $k$-element subset of $[k]\times [n]$,
where tuple $(a_1,\dots,a_k)$ corresponds to 
$\{(1,a_1), \dots, (k,a_k)\}$.
Let $\calT_{n,k}$ denote the set of all $n^k$ subsets 
of this form.
Under the aforementioned correspondence between $\calT_{n,k}$ and $[n]^k$,
disjoint tuples in $[n]^k$ correspond to
disjoint sets in $\calT_{n,k}$.
Thus, a matching in $\calF_1,\dots,\calF_s\subseteq \calT_{n,k}$
is a list of $s$ pairwise disjoint sets $B_1\in\calF_1$, \dots,
$B_s\in\calF_s$.
\subsection{Concentration in Random Matchings}
The key input to both \Cref{thm:with_k}
and \Cref{thm:without_k} is concentration of the intersection 
of random matchings with a fixed set $\calG\subseteq [n]^k$.

A \vocab{matching} in $\calT_{n,k}$ is a collection of pairwise 
disjoint elements of $\calT_{n,k}$.
A matching is \vocab{perfect} if and only if 
it contains $n$ elements.
We consider a uniformly random perfect matching in $\calT_{n,k}$.
Using martingale concentration inequalities, 
Kiselev and Kupavskii \cite{rainbow_matching}
proved the following concentration result.
\begin{theorem}[{Concentration of Random Matching,
    \cite[Thm.\ 6]{rainbow_matching}}]
\label{thm:concentration_matching}
Let $\calM$ be a uniformly random perfect matching in $\calT_{n,k}$.
Let $\calG \subset [n]^k$ be a subset with 
$|\calG| = \alpha n^k$. Then for any $\lambda>0$, we have 
\begin{align*}
\mathbb P\Big(|\calG\cap\calM| \geq \alpha n + 2\lambda\Big)
&\leq 2\exp\left(-\frac{\lambda^2}{\alpha n/2 + 2\lambda}\right)
\quad\text{and} \\
\mathbb P\Big(|\calG\cap\calM| \leq \alpha n - 2\lambda\Big)
&\leq 2\exp\left(-\frac{\lambda^2}{\alpha n/2 + 2\lambda}\right).
\end{align*}
\end{theorem}

We restate the theorem into the following more useful form.
\begin{corollary}
\label{cor:concentration_matching}
Let $\calM$ be a uniformly random matching in $\calT_{n,k}$.
Let $\calG\subset [n]^k$. Then for any $m>0$, we have 
\begin{align*}
\mathbb P\left(|\calG\cap\calM| \geq \frac{|\calG|}{n^{k-1}} 
+ \max\left(2\sqrt{\frac{|\calG|\log(2m)}{n^{k-1}}}, 
8\log(2m)\right) \right)
&< \frac 1m\quad\text{and} \\
\mathbb P\left(|\calG\cap\calM| \leq \frac{|\calG|}{n^{k-1}} 
- \max\left(2\sqrt{\frac{|\calG|\log(2m)}{n^{k-1}}}, 
8\log(2m)\right) \right) & < \frac 1m.
\end{align*}
\end{corollary}
\begin{proof}
Select $\lambda=\max\Big( 
    \sqrt{\frac{|\calG|\log(2m)}{n^{k-1}}}, 4\log(2m)\Big)$.
We apply \Cref{thm:concentration_matching},
and the result follows from 
\begin{align*}
2\exp\left(-\frac{\lambda^2}{\alpha n/2 + 2\lambda}\right)
&\leq 2\exp\left(-\frac{\lambda^2}{2\max(\alpha n/2, 2\lambda)}\right)
\\
&= 2\exp\left(-\min\left(\frac{\lambda^2}{\alpha n}
, \frac{\lambda}4\right)\right) \leq \frac 1m,
\end{align*}
where the last inequality follows from 
$\lambda^2/\alpha n \geq \log(2m)$
and $ \lambda/4\geq \log(2m)$.
\end{proof}

We note the following corollary, which applies to 
multisets $\calG$.
\begin{corollary}
\label{cor:matching_multiset}
Let $\calM$ be a uniformly random perfect matching in $[n]^k$.
Let $\calG$ be a multiset of elements in $[n]^k$ 
such that each element appears in $\calG$ at most $t$ times.
Then for any $m>0$, we have 
\begin{align*}
\mathbb P\left(|\calG\cap\calM| \geq \frac{|\calG|}{n^{k-1}} 
+ \max\left(2t\sqrt{\frac{|\calG|\log(2tm)}{n^{k-1}}}, 
8t\log(2tm)\right) \right)
&\leq \frac{1}{m} \quad\text{and} \\
\mathbb P\left(|\calG\cap\calM| \leq \frac{|\calG|}{n^{k-1}} 
- \max\left(2t\sqrt{\frac{|\calG|\log(2tm)}{n^{k-1}}}, 
8t\log(2tm)\right) \right)
&\leq \frac{1}{m},
\end{align*}

(Here, $|\calG|$ and $|\calG\cap\calM|$ 
counts the size of $\calG$ and $\calG\cap M$ 
with respect to multiplicities of $\calG$.)
\end{corollary}
\begin{proof}
Split $\calG$ into the sum of $t$ sets 
$\calG=\calG_1+\dots+\calG_t$
where $\calG_1,\dots,\calG_t \subseteq [n]^k$.
The conclusion then follows from applying 
the previous theorem on each of the $\calG_i$
(with $m$ replaced by $mt$)
and using a union bound. 
\end{proof}

\section{Shifting Argument: Proof of \texorpdfstring{\Cref{thm:with_k}}
{Theorem \ref*{thm:with_k}}}
\label{sec:with_k}
In this section, we prove \Cref{thm:with_k}.
To do this, we use a shifting argument
to simplify the structure of $\calF_1,\dots,\calF_s$,
which allows us to have better control when picking random matchings.
As explained in \Cref{subsec:notation},
we view $\calF_1,\dots,\calF_s$ as subsets of $\calT_{n,k}$.

\subsection{Shifting Argument}
Shifting is a technique in extremal set theory
that has been used to prove classical results
such as the Erd\H os-Ko-Rado theorem or the Kruskal-Katona theorem
(see \cite{shifting} for a survey on this technique).
It was used by Huang, Loh, and Sudakov 
\cite{rainbow_emc} to study the rainbow Erd\H os matching conjecture.

Suppose that $\calF_1,\dots,\calF_s\subseteq \calT_{n,k}$.
For each $j\in [k]$ and $a,b\in [n]$,
we define the \vocab{shift map} $S_{j,a,b}$,
which takes a subset $\calF\subseteq \calT_{n,k}$
to another subset of $\calT_{n,k}$
obtained by replacing $(j,b)$ in each element of 
$\calF$ by $(j,a)$ 
whenever possible.
More precisely, for each $F\in\calF$, define
$$S_{j,a,b}(F) = \begin{cases}
F \setminus \{(j,b)\} \cup \{(j,a)\} 
& \text{if } (j,b)\in F \text{ and } 
(F \setminus \{(j,b)\} \cup \{(j,a)\}) \notin\calF \\
F & \text{otherwise}
\end{cases}$$
Then we have
$$S_{j,a,b}(\calF) = \{S_{j,a,b}(F) : F\in\calF\}.$$
The key property of the shift map is 
the following.
\begin{proposition}
If $\calF_1,\dots,\calF_s$ has no rainbow matching, 
then $S_{j,a,b}(\calF_1), \dots, S_{j,a,b}(\calF_s)$
also has no rainbow matching.
\end{proposition}
\begin{proof}
Assume for the sake of contradiction that there is a matching 
$B_1\in S_{j,a,b}(\calF_1)$, \dots, $B_s\in S_{j,a,b}(\calF_s)$.
If $B_i\in\calF_i$ for all $i\in[s]$,
then we have a matching of $\calF_1,\dots,\calF_s$.
Thus, we assume that there exists $i$ such that 
$B_i\notin\calF_i$,
which implies that $(j,a)\in B_i$,
so in particular, $i$ is unique.
Therefore, $B_\ell \in \calF_\ell$
for all $\ell\neq i$.
Moreover, $B_i' = B_i\setminus \{(j,a)\} \cup \{(j,b)\}$
is in $\calF_i$ because it is the element that was shifted 
to $B_i$.

We now proceed in cases on whether $(j,b)\in B_\ell$.
\begin{itemize}
\item If $(j,b)\notin B_\ell$
for all $\ell\in [s]$,
then $B_i'$ is disjoint from $B_\ell$ for all $\ell\neq i$.
Thus, $B_1\in\calF_1$, \dots, $B_i'\in\calF_i$,
\dots, $B_s\in\calF_s$ form a rainbow matching of 
$\calF_1,\dots,\calF_s$, a contradiction.
\item Otherwise, there exists $\ell\in [s]$ 
such that $(j,b)\in B_\ell$,
which must be unique. 
Since $B_\ell$ is in both $\calF_\ell$ 
and $S_{j,a,b}(\calF_\ell)$, it follows that 
$B_\ell' \coloneq B_\ell \setminus \{(j,b)\} \cup \{(j,a)\}$
is in $\calF_\ell$ (because otherwise, $S_{j,a,b}(B_\ell)
= B_\ell'$).
Define $B_m' = B_m$ for all $m\notin\{i,\ell\}$.
Then $B_1'\in\calF_1,\dots,B_s'\in\calF_s$ are pairwise 
disjoint, thus forming a rainbow matching in 
$\calF_1,\dots,\calF_s$, a contradiction.
\qedhere
\end{itemize}
\end{proof}
We apply the following shifting operations
to $\calF_1,\dots,\calF_s$ in the following order:
$$\begin{matrix}
S_{j,1,2}, & S_{j,1,3}, & S_{j,1,4}, & \dots, & S_{j,1,n}, \\
& S_{j,2,3}, & S_{j,2,4}, & \dots, & S_{j,2,n}, \\
& & S_{j,3,4}, & \dots, & S_{j,3,n}, \\
& & & \ddots & \vdots \\
& & & & S_{j,n-1,n}.
\end{matrix}$$
In the resulting sets, we have that 
\begin{equation}
\label{eq:shifting_property}
\text{for all }a<b,
\text{ if }F\in\calF_i \text{ then } 
F\setminus\{(j,b)\}\cup\{(j,a)\}\in \calF_i.
\end{equation}
This property is preserved when we apply the shift map 
$S_{j,a,b}$ for all $j\neq 1$.
Thus, we do this shifting sequence for all 
$j\in [k]$ in arbitrary order 
to get that \eqref{eq:shifting_property} holds for all $j$.

\subsection{Proof of \texorpdfstring{\Cref{thm:with_k}}
{Theorem \ref*{thm:with_k}}}
We now prove \Cref{thm:with_k}, which we reproduce below.
\begin{theorem*}
If $n> s$ and
$$c\geq n^{k-1} + \max\left(k^2sn^{k-\frac 32}\sqrt{8\log(2ks)}, 
8kn^{k-1}\log(2ks)\right),$$
then the sequence $f_i = (i-1)n^{k-1}+c$
is satisfying.
\end{theorem*}

Assume for the sake of contradiction that $\calF_1,\dots,\calF_s\in [n]^k$
with $|\calF_i| > (i-1)n^{k-1}+c$ for all $i$ has no matching.
Take an inclusion-maximal counterexample.
By our shifting argument, we assume that \eqref{eq:shifting_property}
holds for all $j\in [k]$.
\begin{lemma}
\label{lem:low_deg}
For each $i\in [s]$, $j\in [k]$, and $a\in [n]$
such that $a\geq s$, if 
$(j,a)\in F$ for some $F\in\calF_i$,
then $F\setminus \{(j,a)\}\cup\{(j,b)\}\in \calF_i$
for \emph{any} $b\in [n]$.
\end{lemma}
\begin{proof}
From \eqref{eq:shifting_property}, we note that 
$F\setminus \{(j,a)\}\cup\{(j,b)\}\in \calF_i$ 
for all $b\in [s]$ already.
Now, let 
$$\calF_i' = \calF_i \cup \big\{F\setminus \{(j,a)\}\cup\{(j,b)\}
: b\in [n]\big\}.$$
We claim that $\calF_1,\dots,\calF_i',\dots,\calF_s$
has no matching, which will imply by inclusion-maximality 
of $\calF_1,\dots,\calF_s$ that $\calF_i=\calF_i'$, 
giving the desired conclusion.
Assume for the sake of contradiction 
that there is a matching $B_1\in\calF_1,
\dots, B_i\in\calF_i',\dots, B_s\in\calF_s$.
If $B_i\in\calF_i$, then we automatically get 
a matching of $\calF_1,\dots,\calF_s$, a contradiction.
Thus, assume $B_i\in\calF_i'\setminus\calF_i$,
which means that $B_i = F\setminus\{(j,a)\}\cup \{(j,b)\}$
for some $b\in [n]$.

Select $b'\in [s]$ such that 
$(j,b')\notin B_\ell$ for all $\ell\neq i$.
This must be possible because each $\ell\neq i$
eliminates at most one possible value of $b'$.
Then we replace $B_i$ with 
$B_i' = F\setminus \{(j,a)\}\cup\{(j,b')\}$,
and then $B_1\in \calF_1, \dots, B_i'\in\calF_i,\dots,
B_s\in\calF_s$ form a rainbow matching in 
$\calF_1,\dots,\calF_s$, a contradiction.
\end{proof}

Pick a uniformly random perfect matching $\calM\subseteq \calT_{n,k}$.
We claim that with high probability,
$|\calM\cap\calF_i| \geq i$.
To prove this, for each $(j,a)$, we define 
$$\calH_{j,a} = \{F\in\calT_{n,k} : (j,a)\in F\}.$$
We also let 
$$T = \{(j,a)\in [k]\times [n] : 
\calH_{j,a}\subseteq \calF_i\}.$$
If $(j,a)\in T$ for some $a\geq s$, then \Cref{lem:low_deg}
gives that $(j,a)\in T$ for all $a\in [n]$,
which implies that $\calF_i = [n]^k$,
so \eqref{eq:high_prob} automatically holds.
Hence, we assume that $a<s$ for all $(j,a)\in T$,
so $|T|<ks$.

We now show that most elements of $\calF_i$ 
lie in some $\calH_{j,a}$ for some $(j,a)\in T$.
This follows from the following claim:
\begin{claim}
\label{claim:small_remainder}
Suppose that $\{(1,x_1),\dots,(k,x_k)\}\in\calF_i$ 
is not in $\calH_{j,a}$ for any $(j,a)\in T$.
Then at least two of $x_1,\dots,x_k$ are in $[s-1]$.
\end{claim}
\begin{proof}
Without loss of generality, assume that $x_1,\dots,x_{k-1}
\notin [s-1]$. 
By \Cref{lem:low_deg} on $(1,x_1)$, 
it follows that for any $y_1\in [n]$, the set 
$\{(1,y_1), (2,x_2), \dots, (k,x_k)\}$ is in $\calF_i$.
Similarly, applying \Cref{lem:low_deg} on $(2,x_2)$,
\dots, $(k-1, x_{k-1})$ gives that 
for any $y_1,\dots,y_{k-1}\in [n]$, we have 
$\{(1,y_1), \dots, (k-1,y_{k-1}), (k,x_k)\}\in\calF_i$.
Hence, $\calH_{k,x_k}\subseteq \calF_i$.
\end{proof}
Let $\calA = \bigcup_{(j,a)\in T}\calH_{j,a}$.
We split $|\calM\cap\calF_i|$ into two terms:
\begin{equation}
\label{eq:concentration_F}
|\calM\cap\calF_i| = |\calM\cap\calA| + |\calM\cap(\calF_i 
\setminus\calA)|.
\end{equation}
To handle the second term, we note that by \Cref{claim:small_remainder},
we have 
$|\calF_i\setminus\calA| \leq \binom k2 s^2n^{k-2}
< \frac{k^2}2 s^2 n^{k-2}$.
Thus, by our concentration result (\Cref{cor:concentration_matching}),
we get that
\begin{equation}
\label{eq:second_term}
\bbP \left(|\calM \cap (\calF_i\setminus\calA)|
\leq \frac{|\calF_i\setminus\calA|}{n^{k-1}}
- \max\left(ks\sqrt{\frac{2\log(2ks)}{n}}, 8\log(2ks)\right)\right)
< \frac{1}{ks}.
\end{equation}
To handle the first term, let $\calB$ be the multiset such that
$$\calA\oplus \calB = \bigoplus_{(j,a)\in T} \calH_{(j,a)},$$
where all sums are calculated as multisets.
Thus, we get that
\begin{align*}
|\calM\cap\calA| 
&= \sum_{(j,a)\in T} |\calM\cap \calH_{(j,a)}| 
- |\calM\cap\calB| \\
&= |T| - |\calM\cap\calB|.
\end{align*}
Each element of $\calA$ appears in $\calB$ at most 
$k-1$ times.
Moreover, since elements of $\calB$ lies in at least 
two sets of the form $\calH_{j,a}$ and at most $ks$ such sets
(since $|T|<ks$), we get that
$|\calB| \leq \binom{ks}2 n^{k-2} < \frac{k^2s^2}2 n^{k-2}$, 
Therefore, by \Cref{cor:matching_multiset}, we have
$$\bbP \left(|\calM \cap \calB|
\geq \frac{|\calB|}{n^{k-1}}
+ \max\left(k(k-1)s\sqrt{\frac{2\log(2ks)}{n}}, 8(k-1)\log(2ks)\right)\right)
< \frac{k-1}{ks}.$$
Thus, combining the previous two equations and noting that
$|\calA| + |\calB| = n^{k-1}|T|$ gives 
\begin{equation}
\label{eq:first_term}
\bbP \left(|\calM \cap \calA|
\leq \frac{|\calA|}{n^{k-1}}
- \max\left(k(k-1)s\sqrt{\frac{2\log(2ks)}{n}}, 8(k-1)\log(2ks)\right)\right)
< \frac{k-1}{ks}.
\end{equation}
Plugging in \eqref{eq:second_term} and \eqref{eq:first_term}
into \eqref{eq:concentration_F} gives
\begin{equation}
\label{eq:high_prob}
\bbP \left(|\calM \cap \calF_i|
\leq \frac{|\calF_i|}{n^{k-1}}
- \max\left(k^2s\sqrt{\frac{8\log(2ks)}n}, 8k\log(2ks)\right)\right)
< \frac 1s.
\end{equation}
Thus, by the union bound, there exists a matching $\calM$
such that for all $i\in [s]$,
\begin{align*}|\calM\cap\calF_i| 
&\geq 
\frac{|\calF_i|}{n^{k-1}}
- \max\left(k^2s\sqrt{\frac{8\log(2ks)}n}, 8k\log(2ks)\right) \\
&\geq (i-1) + \frac{c}{n^{k-1}}
- \max\left(k^2s\sqrt{\frac{8\log(2ks)}n}, 8k\log(2ks)\right),
\end{align*}
which is at least $i$ by our constraint on $c$.
Therefore, $|\calM\cap\calF_i|\geq i$ for all $i$,
and so we can find a matching of $\calF_1,\dots,\calF_s$
by going through $i=1,2,\dots,s$ in order
and picking an element in $\calM\cap\calF_i$
that has not been picked yet.
\section{Spread Approximation: Proof of 
\texorpdfstring{\Cref{thm:without_k}}{Theorem \ref*{thm:without_k}}}
\label{sec:without_k}
In this section, we prove \Cref{thm:without_k}.
The idea is similar to the proof of \Cref{thm:with_k},
but we require a stronger structural theorem on the sets 
to eliminate the factor of $k$.
As before,
we view $\calF_1,\dots,\calF_s$ as subsets of $\calT_{n,k}$.
\subsection{Spread Approximation}
In order to eliminate the factor of $k$ in the bound
and prove \Cref{thm:without_k},
we use the method of spread approximation,
which was first introduced by Kupavskii and Zakharov
in \cite{spread_approx} and was applied in 
the setting of rainbow matchings in the proof of 
\cite[Thm.~9]{satisfying_sequence}.

The idea of spread approximation is to 
approximate $\calF_i$
by a collection $\calS_i$ of subsets of size at most $2$ 
of an element in $\calF_i$
(i.e., subsets of size at most $2$ of $[k]\times [n]$).
We repeatedly take out small subsets that appear unusually frequently 
(i.e., contained in unusually many elements of $\calF_i$)
until we cannot do that anymore.
We can show
(using the spread lemma, discovered by Alweiss,
Lowett, Wu, and Zhang
\cite{sunflower} and sharpened by Tao in \cite[Prop.~5]{spread_lemma})
that the resulting $\calS_i$ has no matching.

We now explain how spread approximation is applied 
to rainbow matchings as used in
\cite[Thm.~9]{satisfying_sequence}.
To do that, we introduce the following notation
used in \cite{spread_approx}:
for any families $\calF$, $\calS$
and set $X$, we define 
\begin{align*}
\calF[X] &\coloneq \{F\in\calF : X\subseteq F\} \\
\calF(X) &\coloneq \{F\setminus X : F\in\calF, X\subseteq F\} \\
\calF[\calS] &\coloneq \bigcup_{A\in\calS} \calF[A].
\end{align*}

The following result is from \cite[Thm.~9]{satisfying_sequence},
except one additional condition, which we explain below.
\begin{theorem}
\label{thm:spread_approx}
Suppose that $n>2^5s\log_2(sk)$.
Let $\calF_1,\dots,\calF_s\subseteq \calT_{n,k}$
have no rainbow matching.
Then there exist families $\calS_1, \dots, \calS_s$ such that 
\begin{enumerate}[label=(\alph*)]
\item \label{item:elements}
One can write $\calS_i = \calS_i^{(0)}\sqcup 
\calS_i^{(1)}\sqcup\calS_i^{(2)}$
where for $t\in\{0,1,2\}$, each element of $\calS_i^{(t)}$
is a $t$-element subset of $[k]\times [n]$;
\item \label{item:union}
(Small leftover) 
If $\calF_i' = \calF_i\setminus \calT_{n,k}[\calS_i]$,
then $|\calF_i'| \leq 2^{15} s^3 \log_2(sk)^3 n^{k-3}$;
\item \label{item:no_matching}
(No rainbow matching) 
One cannot pick $s$ pairwise disjoint sets
$B_1\in \calS_1$, \dots, $B_s\in\calS_s$;
\item \label{item:small_S1}
We have $|\calS_i^{(1)}| \leq 2(s-1)$ for all $i$;
\item \label{item:small_S2}
We have $|\calS_i^{(2)}| \leq 4(s-1)^2$ for all $i$.
\end{enumerate}
\end{theorem}
In comparison, the result in \cite[Thm.~9]{satisfying_sequence}
asserts that there exists $\calS_1,\dots,\calS_s$
satisfying conditions \ref{item:elements}, \ref{item:union},
\ref{item:no_matching}, and \ref{item:small_S2}.
For our applications, we need our $\calS_1,\dots,\calS_s$ 
to satisfy \ref{item:small_S1} as well,
and we prove that below.
For completeness, we also include the proof of parts
\ref{item:elements}, \ref{item:union}, and \ref{item:small_S2}.
The proof of part \ref{item:no_matching} is more technical 
and hence will be omitted.
\begin{proof}
Let $r=2^5 s \log_2(sk)$.
For each $i\in [s]$, we construct $\calS_i$ as follows.
Initialize $\calG_i=\calF_i$ and $\calS_i=\varnothing$.
Then repeat the following steps:
\begin{itemize}
\item Choose an inclusion-maximal $S\subseteq [k]\times [n]$
such that $|\calG_i(S)| \geq r^{-|S|} |\calG_i|$.
This exists since $\varnothing$ works.
\item If $|S|\geq 3$ or $\calG_i=\varnothing$, then stop.
\item Otherwise, add $S$ as an element to $\calS_i$ 
and redefine $\calG_i$ to be $\calG_i\setminus\calG_i[S]$.
\end{itemize}
The process finishes when either $\calG_i=\varnothing$
or we find a set $S\subseteq [k]\times [n]$ 
such that $|S|\geq 3$ and $|\calG_i(S)| 
\geq r^{-|S|}|\calG_i|$. 
We claim that at this point, conditions \ref{item:elements}, 
\ref{item:union}, and \ref{item:no_matching} are satisfied.
 We verify these.
\begin{enumerate}[label=(\alph*)]
\item By construction, we always add sets of sizes 
$0$, $1$, or $2$.
\item By construction, we have 
$\calF_i \supseteq \calT_{n,k}[\calS_i] \cup \calG_i$ 
at any point in the algorithm.
Thus, in the final stage, $\calG_i\supseteq \calF_i'$.
If $\calG_i=\varnothing$, then the claim is clear.
Otherwise, we have 
$$|\calF_i'| \leq |\calG_i| 
\leq r^{|S|} |\calG_i(S)| 
\leq r^{|S|} n^{k-|S|} \stackrel{(*)}\leq r^3n^{k-3}
= 2^{15}s^3\log_2(sk)^3 n^{k-3},$$
where the inequality marked ($*$)
follows from the assumption $r<n$.

\item See the proof of 
\cite[Thm.~9]{satisfying_sequence} for details.
\end{enumerate}
Now, we will modify $\calS_i$ so that \ref{item:small_S2} holds,
and then so that \ref{item:small_S1} holds (and all other conditions
are preserved).
\begin{enumerate}
\item[(e)] First, we modify $\calS_1,\dots,\calS_s$ so that 
each element $(j,a)\in [k]\times [n]$ 
appears in at most $2(s-1)$ sets in $\calS_i^{(2)}$.
Assume that the element $(j,a)$ appears in $\calS_i^{(2)}$ 
at least $2s-1$ times.
Then we add $\{(j,a)\}$ to $\calS_i^{(1)}$
and remove every element containing $(j,a)$ from $\calS_i^{(2)}$.
Let the resulting set be $\calS_i'$.
This modification clearly preserves \ref{item:elements} and \ref{item:union},
so we have to check \ref{item:no_matching}, i.e., it does not create 
additional matchings.
Suppose that there is a matching 
$B_1\in\calS_1$, \dots, $B_i=\{(j,a)\}\in\calS_i'$, \dots, 
$B_s\in\calS_s$. 
The union $\bigcup_{j\neq i} B_j$ has at most $2s-2$ elements,
so there exists an element $B_i'\in \calS_i^{(2)}$
that contains $(j,a)$ and does not intersect this union.
Replacing  $B_i$ with $B_i'$ gives a matching in the original
$\calS_1,\dots,\calS_s$, which contradicts \ref{item:no_matching}.

Next, if $|\calS_i^{(2)}| > 4(s-1)^2$,
then we may replace $\calS_i$ with $\{\varnothing\}$.
Again, the only nontrivial item to check is \ref{item:no_matching}.
Suppose there is a matching 
$B_1\in\calS_1$, \dots, $\varnothing\in\calS_i'$,
\dots, $B_s\in\calS_s'$.
From the previous paragraph, for each $j\neq i$,
there are at most 
$2\cdot 2(s-1)$ elements in $\calS_i^{(2)}$
that meet $B_j$.
Thus, there are at most $4(s-1)^2$ elements in 
$\calS_i^{(2)}$ that intersect $B_j$ for some $j\neq i$,
which means that there exists an element $B_i'\in\calS_i^{(2)}$
that does not intersect $B_j$ for all $j\neq i$,
so replacing $B_i$ with $B_i'$ gives a matching of 
$\calS_1,\dots,\calS_s$, which contradicts \ref{item:no_matching}.
\item[(d)] 
Assume for contradiction that 
$|\calS_i^{(1)}| \geq 2s-1$. 
Then we claim that we may replace $\calS_i$ with $\{\varnothing\}$.
This modification preserves
\ref{item:elements}, \ref{item:union},
and \ref{item:small_S2},
so we have to check that it preserves \ref{item:no_matching}.
Suppose that there is a matching 
$B_1\in\calS_1,\dots,\varnothing\in\calS_i,\dots,B_s\in\calS_s$.
Then one may replace $B_i$ with an element in 
$\calS_i^{(1)}$ not contained in $\bigcup_{j\neq i}B_j$,
which must exist because 
$\left|\bigcup_{j\neq i} B_j\right| \leq 2(s-1)$.
This gives a matching in $\calS_1,\dots,\calS_s$,
which contradicts \ref{item:no_matching}.
\qedhere 
\end{enumerate}
\end{proof}
\subsection{Proof of \texorpdfstring{\Cref{thm:without_k}}
{Theorem \ref*{thm:without_k}}}
We now prove \Cref{thm:without_k}, which we reproduce 
the statement below.
\begin{theorem*}
If $n > \max(2^5 s\log_2(sk), s)$ and 
$$c \geq n^{k-1} + \max\left(14sn^{k-\frac 32}\sqrt{\log(2ks)},
8kn^{k-1}\log(2ks)\right) + 2^{15} s^3\log_2(ks)^3 n^{k-3},$$
then the sequence $f_i = (i-1)n^{k-1}+c$ is satisfying.
\end{theorem*}
Assume for the sake of contradiction that $\calF_1,\dots,\calF_s$
has no matching.
Let $\calS_1,\dots,\calS_s$ be as in \Cref{thm:spread_approx}.
We define 
$$u\coloneq s\sqrt{\frac{\log ks}n}.$$

Property \ref{item:no_matching} implies that one 
cannot find a rainbow matching in 
$(\calT_{n,k}[\calS_i])_{i=1}^s$
(because if there is a rainbow matching in $\calT_{n,k}[\calS_1]$,
\dots, $\calT_{n,k}[\calS_s]$,
then one can find a rainbow matching 
in $\calS_1,\dots,\calS_s$ by selecting corresponding subsets,
contradicting \ref{item:no_matching}.)
We now find a matching in $(\calT_{n,k}[\calS_i])_{i=1}^s$
to obtain a contradiction.
To do this, let $\calM\subseteq \calT_{n,k}$ 
be a uniformly random perfect matching of $[n]^k$
(so $|\calM|=n$). 

We claim that with high probability, 
$|\calM\cap\calT_{n,k}[\calS_i]|\geq i$.
This clearly always holds when $\calS_i$ contains $\varnothing$.
Thus, we assume that $\varnothing\notin\calS_i$.
We let 
$$\calU_1 := \calT_{n,k}\big[\calS_i^{(1)}\big] 
\quad\text{and}\quad \calU_2 :=
\calT_{n,k}\big[\calS_i^{(2)}\big]
\setminus \calU_1.$$
We then break the expression 
$|\calM\cap \calT_{n,k}[\calS_i]|$ into two terms:
\begin{equation}
\label{eq:concentration_AS}
|\calM\cap\calT_{n,k}[\calS_i]| 
= |\calM\cap \calU_1|
+ |\calM \cap \calU_2|.
\end{equation}
To handle the second term,
we use property \ref{item:small_S2} to get that 
$|\calU_2| \leq 
\big|\calT_{n,k}\big[\calS_i^{(2)}\big]\big| < 4s^2 n^{k-2}$,
so by our concentration result, \Cref{cor:concentration_matching}, 
we find that 
\begin{equation}
\label{eq:concentration_AS2}
\bbP \left(\left|\calM \cap \calU_2\right|
\leq \frac{|\calU_2|}{n^{k-1}}
- \max(4u, 8\log(2ks))\right)
< \frac{1}{ks}.
\end{equation}

We now consider the more difficult first term 
$|\calM\cap\calU_1|$.
We write $\calS_1^{(1)}$ as the union 
$\calA_1\cup \dots\cup \calA_k$, 
where for each $i$,
$\calA_i$ only contains sets of the form 
$\{(i,a)\}$ for $a\in [n]$.
We also let $a_i = |\calA_i|$.
Thus, $\calU_1 = \bigcup_{i=1}^k \calT_{n,k}[\calA_i]$.
We define
$$\calB_\ell = \{F\in\calU_1
: F\in\calT_{n,k}[\calA_j]
\text{ for at least } \ell 
\text{ values of } j\}.$$
For all $i\geq 2$, we can bound the size of $\calB_\ell$ by 
\begin{align*}
|\calB_\ell| &\leq \sum_{1\leq j_1 < \dots < j_\ell \leq k}
|\calT_{n,k}[\calA_{j_1}] \cap \dots \cap \calT_{n,k}[\calA_{j_\ell}]| \\
&= \sum_{1\leq j_1 < \dots < j_\ell \leq k}
a_{j_1} \cdots a_{j_\ell} n^{k-\ell} \\
&\leq \frac 1{\ell!}(a_1+\dots+a_k)^\ell n^{k-\ell} \\
&<  \frac 1{\ell!} (2s)^\ell n^{k-\ell}
\tag{property \ref{item:small_S1}} \\
&\leq \frac{2^\ell}{\ell!} s^2 n^{k-2}.
\tag{$\ell\geq 2$ and $s\leq n$}
\end{align*}
Thus, by our concentration result, \Cref{cor:concentration_matching}, we have that 
for $\ell\geq 2$,
\begin{equation}
\label{eq:concentration_Bi}
\mathbb P\left(|\calM\cap\calB_\ell| 
\geq \frac{|\calB_\ell|}{n^{k-1}}
+ \max\left(2\sqrt{\frac{2^\ell}{\ell!}}\cdot u, 8\log(2ks)\right)
\cdot u\right) < \frac 1{ks}.
\end{equation}
Next, we use the following fact:
if $C_1,\dots,C_k$ are sets and 
$D_\ell$ is the set of elements appearing in at least $\ell$
of the $k$ sets $C_1,\dots,C_k$, then 
$$\left|\bigcup_{\ell=1}^k C_\ell\right| 
= \left(\sum_{i=\ell}^k |C_\ell|\right) - |D_2| - |D_3| - 
\dots - |D_k|.$$
This fact implies that
\begin{align*}
|\calM\cap\calU_1| 
&= \left(\sum_{i=1}^k |\calM\cap\calT_{n,k}[\calA_i]|\right)
- |\calM\cap\calB_2| - \dots - |\calM\cap\calB_k| \\
&= (a_1+\dots+a_k) 
- |\calM\cap\calB_2| - \dots - |\calM\cap\calB_k| \\
&= |\calS_i^{(1)}| - 
|\calM\cap\calB_2| - \dots - |\calM\cap\calB_k|.
\end{align*}
Combining this with \eqref{eq:concentration_Bi},
using the union bound, and
noting that $\sum_{\ell=2}^\infty \sqrt{2^\ell/\ell!}
< 4.5053 < 5$ gives
\begin{equation}
\label{eq:concentration_AS1}
\mathbb P\left(|\calM\cap\calU_1| 
\leq \frac{|\calU_1|}{n^{k-1}}
- \max(10u, 8(k-1)\log(2ks)) \right)
< \frac{k-1}{ks}.
\end{equation}

Finally, plugging \eqref{eq:concentration_AS2}
and \eqref{eq:concentration_AS1}
into \eqref{eq:concentration_AS}
and using the union bound gives 
$$\mathbb P\left(|\calM\cap\calT_{n,k}[\calS_i]| 
\leq \frac{|\calT_{n,k}[\calS_i]|}{n^{k-1}}
- \max(14u, 8k\log(2ks))\right)
< \frac 1s.
$$
Thus, by the union bound, there exists a matching $\calM$
for which the above event does not happen for any $i$.
In other words, for all $i$, we have
\begin{align*}
|\calM\cap\calT_{n,k}[\calS_i]| 
&\geq \frac{|\calT_{n,k}[\calS_i]|}{n^{k-1}}
- \max(14u, 8k\log(2ks)) \\
&\geq \frac{|\calF_i| - |\calF_i'|}{n^{k-1}}
- \max(14u, 8k\log(2ks))
\tag{property \ref{item:union}}\\
&\geq (i-1) + \frac{c}{n^{k-1}} - 
\max(14u, 8k\log(2ks)) 
- \frac{2^{15}s^3\log_2(sk)^3}{n^2},
\end{align*}
so we have that $|\calM\cap\calT_{n,k}[\calS_i]| \geq i$
for all $i$.
By going through $i=1,2,\dots,s$ in order
and picking an element in $\calM\cap\calT_{n,k}[\calS_i]$
not previously used, we can find a rainbow matching in 
$(\calT_{n,k}[\calS_i])_{i=1}^s$ as desired.
This proves \Cref{thm:without_k}.

\section{Polynomial Method}
\label{sec:polynomial_method}
In this section, we prove \Cref{thm:polynomial_method}.
The key tools we will use are the Schwartz--Zippel lemma
\cite{schwartz}
and the multivariate generalization of combinatorial 
nullstellensatz by Do\u gan, Er\"gur, Mundo,
and Tsigaridas \cite{multivar_null}.
\begin{lemma}[{Schwartz--Zippel Lemma, 
\cite[Lem.~1]{schwartz}}]
\label{lem:schwartz_zippel}
Let $f\in\mathbb Q[x_1,\dots,x_n]$ be a polynomial of degree $d$.
Let $S$ be a finite subset of $\mathbb Q$.
Then,
$$\big|\{(a_1,\dots,a_n)\in S^n : f(a_1,\dots,a_n)=0\}\big|
\leq d |S|^{n-1}.$$
\end{lemma}

In order to state the multivariate combinatorial nullstellensatz,
we make the following definition.
\begin{definition}
For any set $S\subset\mathbb Q^n$, define 
\begin{align*}
I(S) & \coloneq  \{f\in\mathbb Q[x_1,\dots,x_n] : f(\vec x)=0
\text{ for all }\vec x\in S\} \\
\deg S & \coloneq 
\min_{\substack{f\in I(S) \\ f\neq 0}} 
\deg f.
\end{align*}
\end{definition}
\begin{theorem}[generalization of {\cite[Thm.~1.4]{multivar_null}}]
\label{thm:multivar_combonull}
Let $f$ be a polynomial in $n_1+\dots+n_m$ variables
$x_{11}$, $\dots$, $x_{1n_1}$, $\dots$,
$x_{m1}$, $\dots$, $x_{mn_m}$
with coefficients in $\mathbb Q$.
For each $i\in [m]$ and $j\in [n_i]$,
let $e_{ij}\geq 0$ be an integer such that
$\sum_{i=1}^m \sum_{j=1}^{n_i} e_{ij} = \deg f$ and
the coefficient of $\prod_{i=1}^m \prod_{j=1}^{n_i} x_{ij}^{e_{ij}}$
in $f$ is nonzero.
Let $d_i = \sum_{j=1}^{n_i} e_{ij}$.
For each $i\in [m]$, let
$S_i\subseteq \mathbb Q^{n_i}$ such that $\deg S_i > d_i$.

Then there exist $\vec s_i\in S_i$
such that $f(\vec s_1,\dots,\vec s_m)\neq 0$.
\end{theorem}

This is a more general version of \cite[Thm.~1.4]{multivar_null}.
In their version, they have the condition that 
$\sum_{j=1}^{n_i}e_{ij}$ is equal to the degree in variables 
$x_{i1},\dots, x_{in_i}$ for all $i$.
We relax this to only that the sum of $e_{ij}$'s 
is equal to the degree of $f$.

The proof of \Cref{thm:multivar_combonull} is more or less 
the same as the proof of \cite[Thm.~1.4]{multivar_null}.
We only sketch the proof, highlighting the key differences.
\begin{proof}[Proof Sketch of \Cref{thm:multivar_combonull}]
We use the following multi-index notation:
for any $\vec s\in \mathbb F^n$ and $\vec f\in \mathbb Z_{\geq 0}^n$,
we let $\vec s\,^{\vec u} = s_1^{u_1}\cdots s_n^{u_n}$.
For each $\vec f\in \mathbb Z_{\geq 0}^n$, we let 
$\Sigma(\vec u) = u_1+\dots+u_n$.

In the proof of \cite[Thm.~1.4]{multivar_null},
they establish functions $g_i : S_i\to \mathbb F$
such that for all $\vec u = (u_{i1}, u_{i2},\dots, u_{in_i})\in \mathbb Z_{\geq 0}^{n_i}$
such that $u_{i1}+\dots + u_{in_i} \leq d_i$,
we have that 
$$\sum_{\vec s
= (s_1,\dots,s_{n_i})\in S_i} 
g_i(\vec s)\ \vec s\,^{\vec u}
= \begin{cases}
0 & \text{if } (u_{i1}, \dots, u_{in_i}) \neq 
(e_{i1}, \dots, e_{in_i}) \\
1 & \text{otherwise}.
\end{cases}$$
Establishing this uses a linear algebra argument,
which we refer the reader to \cite[p.~15]{multivar_null}.
The argument crucially relies on the assumption that 
$\deg S_i>d_i$.

Once we have this, we assume for contradiction that 
$f(\vec s_1,\dots,\vec s_m) = 0$ for all $\vec s_i\in S_i$.
We consider the sum 
$$T= \sum_{(\vec s_1,\dots,\vec s_m) \in S_1\times\dots\times S_m}
g_1(\vec s_1) \cdots g_m(\vec s_m) \,
f(\vec s_1,\dots,\vec s_m).$$
From what we assumed for contradiction, we have $T=0$.
However, if we consider one monomial 
$\vec x_1^{\vec u_1} \cdots \vec x_m^{\vec u_m}$,
then the sum can be factored as 
$$\prod_{i=1}^m \left(\sum_{\vec s_i\in S_i} g_i(\vec s_i) 
\vec{s_i}^{\vec u_i}\right).$$
If $\Sigma(\vec u_i) < d_i$ for some $i$, then the sum above 
is immediately zero.
Otherwise, $\Sigma(\vec u_i) \geq d_i$ for all $i$.
However, adding this across all $i$ gives that 
$\Sigma(\vec u_1)+\dots+\Sigma(\vec u_m) \geq \deg f$,
so the degree bound forces all inequalities to be equality.
In particular, $\Sigma(\vec u_i) = d_i$ for all $i$.
If $\vec u_i \neq (e_{i1},\dots,e_{in_i})$ for some $i$,
then again, the sum vanish.
Thus, the only monomial that does not vanish occurs when 
$\vec u_i = (e_{i1},\dots,e_{in_i})$ for all $i$,
in which case the sum is $1$.
Therefore, $T$ is equal to the coefficient of 
$\prod x_{ij}^{e_{ij}}$, which is nonzero by the condition 
of the theorem. This is a contradiction.
\end{proof}

We now prove \Cref{thm:polynomial_method},
which we reproduce the statement below.
\begin{theorem*}
Let $k=2$, and
let $(a_1,\dots,a_s)$ and $(b_1,\dots,b_s)$ be two permutations of 
$\{0,1,\dots,s-1\}$.
Then the sequence $f_i = n(a_i+b_i)$ is satisfying.
\end{theorem*}
\begin{proof}[Proof of \Cref{thm:polynomial_method}]
Let $\calF_1,\dots,\calF_s\subseteq [n]^2$
with $|\calF_i| > n(a_i+b_i)$. 
Consider the polynomial 
\begin{align*}
f(x_1,y_1,\dots,x_s,y_s) &= \prod_{1\leq i<j\leq s} 
(x_i-x_j)(y_i-y_j) \\
&= \left(\sum_{\sigma} \operatorname{sign}(\sigma) 
x_1^{\sigma(1)-1} \cdots x_s^{\sigma(s)-1} \right)
\left(\sum_{\tau} \operatorname{sign}(\tau) 
y_1^{\tau(1)-1} \cdots y_s^{\tau(s)-1} \right),
\end{align*}
where the sum runs through permutations $\sigma$ and $\tau$
of $[s]$, and the last equality follows from
Laplace expansions of Vandermonde determinants.
This gives that the coefficient of 
$x_1^{a_1}\cdots x_s^{a_s} y_1^{b_1}\cdots y_s^{b_s}$
is $\pm 1$, depending on the sign of permutations $(a_i)$
and $(b_i)$.
This verifies that the relevant coefficient is nonzero.

For any  polynomial $g$ that vanishes 
on $\calF_i$, we have 
$\calF_i \subseteq \{(x,y)\in [n]^2 : g(x,y)=0\}$,
so by Schwartz--Zippel lemma (\Cref{lem:schwartz_zippel}),
$$|\calF_i| \leq \deg g\cdot n^{2-1} = n\cdot \deg g,$$
Therefore, $\deg\calF_i > a_i+b_i$. 
Thus, by \Cref{thm:multivar_combonull}, there exists 
$(x_i,y_i)\in\calF_i$ for each $i\in [s]$ such that
$f(x_1,y_1,\dots,x_s,y_s) \neq 0$,
which means that $x_1,\dots,x_s$
and $y_1,\dots,y_s$ are pairwise distinct.
Hence, tuples $(x_i,y_i)$ yield the desired matching.
\end{proof}
\printbibliography
\end{document}